\documentclass[11pt]{article}
\usepackage{amsfonts, amsmath}
\usepackage{amssymb}
\usepackage{amsthm,amscd}
\usepackage{enumerate}
\usepackage{tikz}
\usepackage{lipsum}
\usepackage{subfig}
\usepackage{lineno}
\usepackage{doi}
\usepackage{caption}

\newtheorem{theorem}{Theorem}[section]
\newtheorem{corollary}[theorem]{Corollary}

\newtheorem{conjecture}[theorem]{Conjecture}

\numberwithin{equation}{section}
\numberwithin{figure}{section}

\newcommand{\R}{\mathbb R}

\newtheorem{lemma}{Lemma}%[equation]
\theoremstyle{definition}
\newtheorem*{remark}{Remark}%[equation]
\newtheorem{definition}{Definition}%[equation]

\def\p{{\bf p}}

\def\q{{\bf q}}

\usepackage[margin=1in]{geometry}

\author{Robert Connelly\thanks{Department of Mathematics, Cornell University, Ithaca, USA.} and Zhen Zhang\thanks{Yau Mathematical Sciences Center, Tsinghua University, Beijing, China.}}
\begin{document}
\title{Spiderwebs on the Sphere and an Isoperimetric Theorem}
\maketitle
\section{Introduction} \label{section:introduction}
Recently a result has been announced by Martin Winter \cite{Winter} that when one takes the cone over the vertices of a convex polytope in three-space from the interior of the polytope, the resulting framework is rigid. In other words, there is no continuous motion of the configuration preserving each edge length as well as the distance to the cone point. Indeed, the structure is rigid with the weaker constraint that each cone length can get no shorter (a strut constraint), and each edge length of the polytope can get no longer (a cable constraint). And, a similar result holds in all dimensions greater than three.

Here we present a stronger result in a more global (semi-global) setting, but for a much more restrictive class of polytopes, those that can be inscribed in a sphere, with some additional conditions, Theorem \ref{Theorem:tensegrity} in this paper. We show that it follows from classical isoperimetric theorems about curves on the unit sphere enclosing the maximum area on a hemisphere with a perimeter of given length (often called the ``isoperimetric inequality").  

The idea is inspired by results in the rigidity of tensegrity structures, where some vertices are pinned, but all the edges are cables that cannot be increased in length.  The structures discussed here look like spiderwebs, which, in this case, enclose a sphere, but cannot ``slip off'', and, when tightened, at least some part of it is rigid. We conjecture that the uniqueness of the spider web, in the semi-global setting, is also true, regardless of the circumscribed condition of the polytope.

\thanks{Thanks to ICERM at Providence from the National Science Foundation under Grant No. DMS-1929284 while the authors were in residence at the institute during the Geometry of Materials, Packings and Rigid Frameworks Conference Jan 29 - May 2, 2025, and for giving us the chance to spread the word about these results.}

\section{Review of Classic Isoperimetric Results}

There are many proofs of the isoperimetric inequality in the plane, which says that any simple closed curve with a fixed length in the plane bounds the most area when it is a circle. First, we state the analogous theorem for the surface of the $2$-dimensional sphere $\mathbb{S}^2$ with unit radius.   Later, we show the connection to Winter's problem.
\begin{lemma}\label{hemi}  If the length $l$ of a curve $C \subset \mathbb{S}^2$  is less than $2\pi$, then $C$ is strictly contained in an open hemisphere of $\mathbb{S}^2$.  
\end{lemma} 

\begin{proof}
    Let $F_1$ and $F_2$ be two points on the curve that are exactly halfway along the curve $C$ in $\mathbb{S}^2$.  The shortest geodesic distance between $F_1$ and $F_2$ is less than $\pi$.  Let $N$ (the north pole) be the midpoint in $\mathbb{S}^2$ along the unique geodesic arc between $F_1$ and $F_2$, and let $H$ be the open half sphere whose center is $N$, and let $E$ (the equator) be the boundary of $H$, which is a geodesic circle in $\mathbb{S}^2$.  Let $p$ be any other point on $C$. Let $a$ be the distance between $p$ and $F_1$, while $b$ is the distance between $p$ and $F_2$. Note that $a+b<\pi$ since the corresponding lengths in $C$ are in an open curve of length at most $\pi$. See Figure \ref{fig:iso}.
\begin{figure}[!htb]
    \centering
        \includegraphics[width=0.4\textwidth]{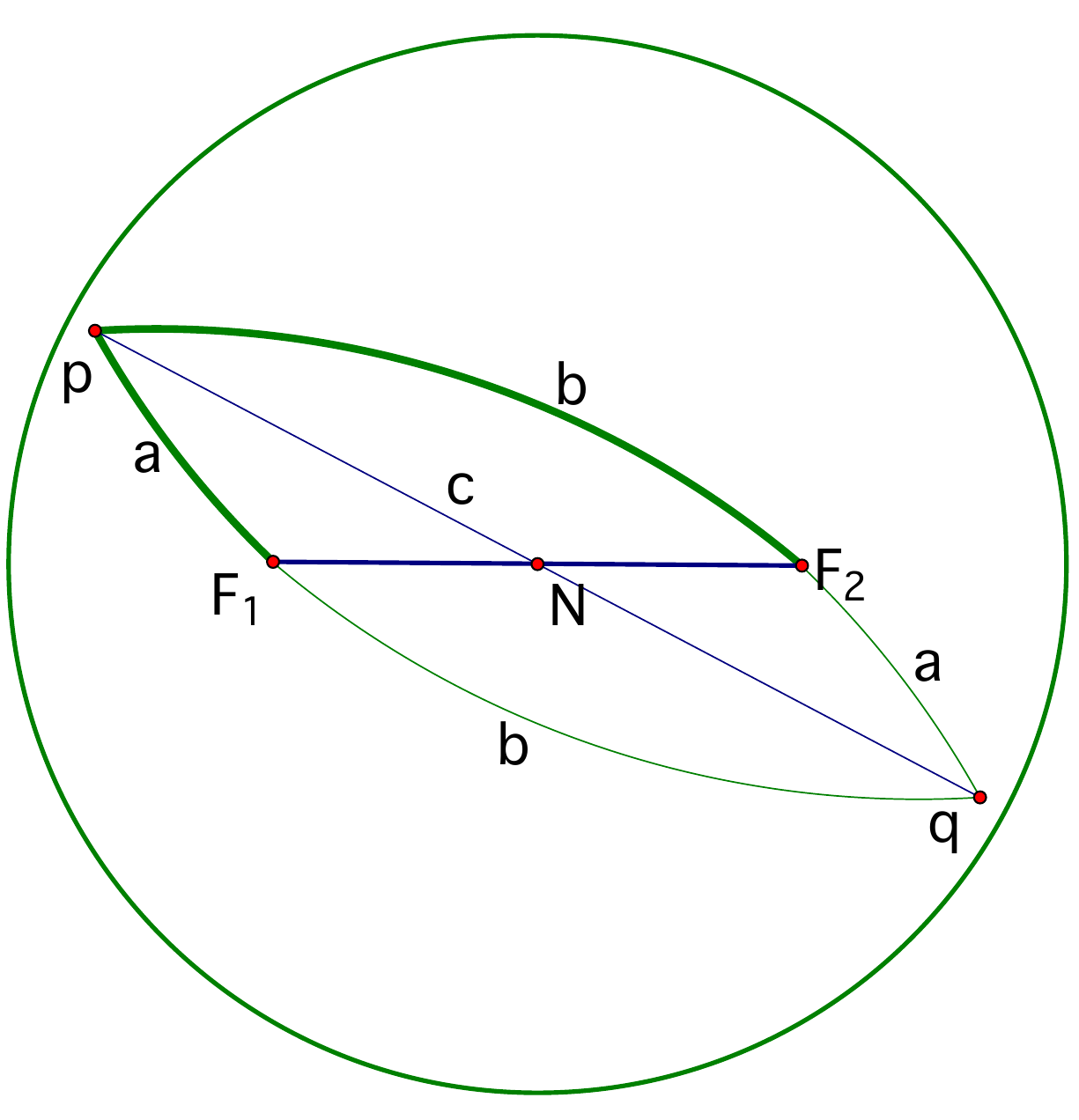}% 
        \captionsetup{labelsep=colon,margin=2cm}
         \caption{This is  a stereographic projection into the plane of the points and line segments defined in the proof of Lemma \ref{hemi} and Lemma \ref{lemma:closed-hemi}. Circles in $\mathbb{S}^2$ are projected into circles and lines in the plane.  The south pole (the point at infinity) is the only point in the sphere not projected to a point in the plane. }\label{fig:iso}
    \end{figure}

The idea is to show that the point $p$ lies in the open half space centered at the point $N$.  All the side lengths of the triangle $F_1, F_2, p$ have geodesic length less than $2\pi$.  Rotate this triangle $180^{\circ}$ about the north pole $N$ to get a congruent triangle $F_2, F_1, q$. So the points $p,N,q$ lie on a geodesic circle and the points $p, F_1, q$ form another triangle with side lengths $a, b, c$.  For fixed points $F_1, F_2$, the points $p$ form a spherical analogue to an ellipse in the plane with foci $F_1, F_2$.  When $p$ lies on the geodesic through $F_1, N, F_2$, then it is in the hemisphere $H$. Since $a+b < 2\pi$, by the spherical triangle inequality, they cannot lie on the equator $E$.  Thus, all the points $p$ must lie in $H$. \qed
\end{proof}
Using the same ideas as in Lemma \ref{hemi}, the following is also useful.  See Figure \ref{fig:iso} as well for this extension.

\begin{lemma}\label{lemma:closed-hemi} If the length $l$ of a curve $C \subset \mathbb{S}^2$ with end points $F_1$ and $F_2$, not antipodal, is such that $ l \le \pi$, where the mid point $N$ of the geodesic between $F_1$ and $F_2$ is the center of a hemisphere $H$, then any point $p \in C$ is contained in the closed hemisphere $H$.

\end{lemma}

\bigskip

In the following, we assume the following isoperimetric theorem for continuous curves of length less than $2\pi$ in $\mathbb{S}^2$, the sphere of unit radius.  We have the tools to prove it for a quadrilateral, for example, and this proves the continuous case, but we will leave that aside here.

\begin{theorem}\label{classic-iso} Any closed continuous curve $C \subset \mathbb{S}^2$ of length $l \le 2\pi$ bounds the largest area in a containing half-sphere in $\mathbb{S}^2$ if and only if $C$ is a circle.
    \end{theorem}
%Actually, that maximum area is $2\pi(1-\cos r)$, but this is not important for the time being.

\begin{lemma}\label{fixed}If a piecewise geodesic curve $C\subset \mathbb{S}^2$ has its vertices lying, in order, on the boundary of a circular disk of radius $r< \pi/2$, then among all curves with those same fixed lengths, the maximum area is achieved when and only when the vertices lie on a circle of the same radius $r$.    
\end{lemma}
\begin{proof} Note that the length of the piecewise geodesic curve is less than $2\pi$ as well, since the geodesic arc lengths are less than the corresponding circular lengths. 
Think of the circular arcs as being attached to the corresponding arcs rigidly with fixed edge lengths, as in Figure \ref{fig:Attachments}.  Then any other configuration of the piecewise geodesic  curve gives a corresponding configuration of the circular arcs, and there is a fixed area difference between the two curves. Then Theorem \ref{classic-iso} applies.\qed

This idea has been used, at least in the planar setting, in \cite{Niven-maxima}, for example. Effectively, the planar isoperimetric theorem, for polygons of fixed edge lengths, is equivalent to the continuous setting.  See the fuss about Steiner's proof, for example \cite{blaasjo2005isoperimetric}. 

\begin{figure}[!htb]
    \centering
    \includegraphics[width=0.5\textwidth]{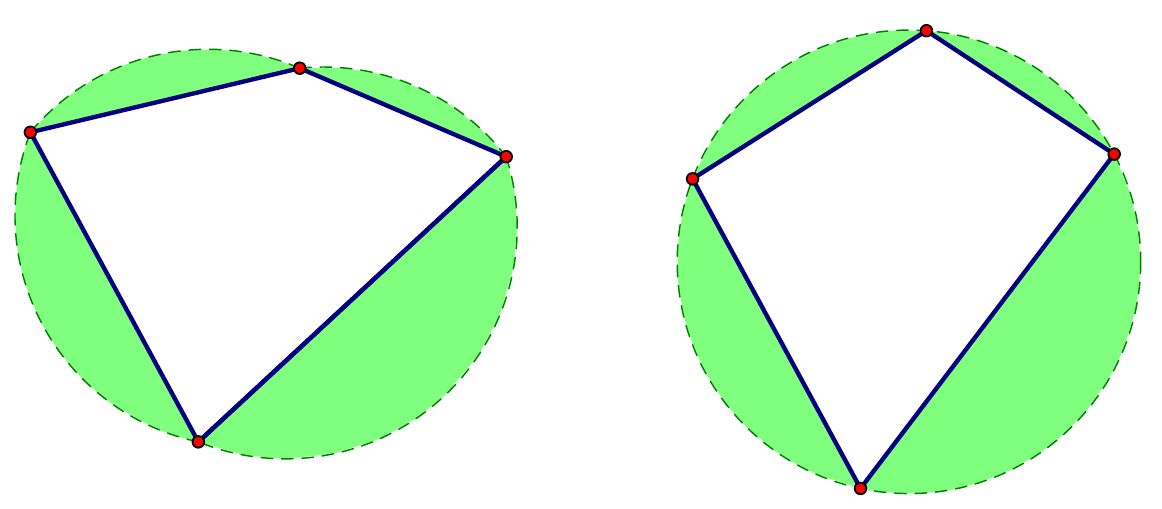}\captionsetup{labelsep=colon,margin=2cm}
         \caption{The green circular segments, the regions between the spherical geodesic segments and the bounding circle, are attached rigidly to the spherical polygon.  This is a picture of the analogous situation in the plane.
         }\label{fig:Attachments}
    \end{figure}
    
\end{proof}

 \section{The Tensegrity Case: Edge Lengths are Cables}   

In the classical continuous case, when the length of the individual edges of the curve become smaller (e.g. infinitesimally small), it is clear that the corresponding maximal area enclosed becomes smaller.  But for the piecewise-geodesic case, that is not always true.  It  depends on the position of the center of the circumcircle (that is the circle on the sphere through some three points)  of the vertices of the maximal area geodesic polygon, even for the case of a triangle.  For all geodesic polygons, all the geodesic arcs are the shortest arc, and thus less than $\pi$.  We will not consider pairs of vertices of the spherical polygon that are antipodal, with the exception of Lemma \ref{lemma:closed-hemi}.  Furthermore, we will call a closed spherical geodesic polygon \textit{proper} if each edge length is strictly less than $\pi$, the vertices do not lie on a great circle, and the sum of its edge lengths is strictly less than $2\pi$. Thus, by Theorem \ref{classic-iso} such a geodesic polygon is contained in an open hemisphere.  Furthermore, we say that such a proper geodesic polygon is \emph{internally circumscribed} if its vertices lie on a circle whose center lies in the closed polygon, and we say that the proper circumscribed geodesic polygon is \emph{strictly internally circumscribed} if the center of the polygon lies in the interior of the geodesic polygon. By Lemma \ref{hemi}, that boundary image is contained in an open hemisphere of $\mathbb{S}^2$. Thus one can extend the map of the boundary of the face to its interior, uniquely, topologically up to the homotopy class of the topological face.  All this together defines a continuous map from the surface of the polytope $P$ to $\mathbb{S}^2$.

    We call a proper circumscribed spherical polygon a \emph{Dido} polygon if the circumcenter lies in the (necessarily the midpoint) of one of the edges of the polygon. We call that longest edge, with the circumcenter in the middle, the \emph{Dido edge}.  This is in honor of Queen Dido of Carthage who solved an isoperimetric problem by the sea.  

    \begin{figure}[!htb]
    \centering
    \includegraphics[width=0.5\textwidth]{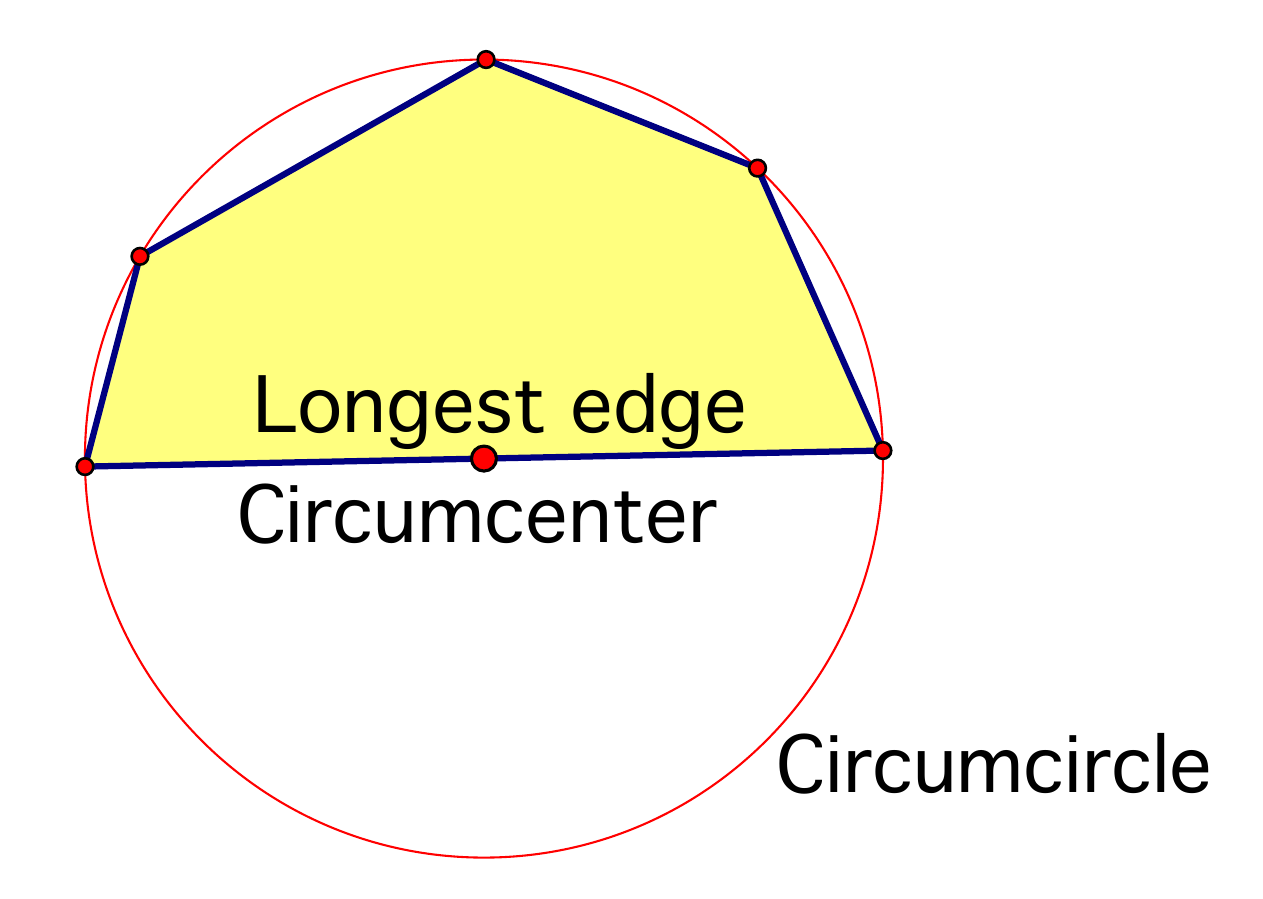}\captionsetup{labelsep=colon,margin=2cm}
         \caption{A Queen Dido polygon: If any single edge, but not the longest, is increased/decreased the area is increased/decreased. When the longest edge of the polygon is changed, fixing the rest, the area of the polygon is decreased.
         }\label{fig:Dido}
    \end{figure}

\begin{theorem}\label{Thm:main} In the space of proper circumscribed geodesic spherical polygons with only one edge allowed to vary, the area function is at a critical point, when and only when the edge is the Dido edge of a Dido polygon. In particular, increasing any non-Dido edge increases the area of the polygon, while decreasing the length of the longest edge, corresponding to when the circumcenter is strictly outside the polygon, increases the area of the polygon.
\end{theorem}

We first mention the result that leads to the rigidity result.

\begin{corollary}\label{lemma:congruence}
If a proper spherical geodesic circumscribed  polygon $P$ has its circumcenter inside or on the boundary of $P$, and $Q$ is any other geodesic polygon, with corresponding edges no longer than the edges of $P$, then the area of $Q$ is no larger than the area of $P$, with equality if and only if $Q$ is congruent to $P$.
\end{corollary}
\begin{proof}
From compactness of the space of polygons, we know that, given the constraints of the polygon $P$, there must be a spherical geodesic polygon $Q$ with maximum area $Q$. By Lemma \ref{fixed}, $Q$ must be circumscribed as well.  If the circumcenter of $Q$ lies outside of $Q$, then the longest edge of $Q$ can be decreased to strictly increase its area, contradicting its optimality. Thus we assume that the circumcenter lies inside or on the boundary of $Q$.  If any edge, not the Dido edge, is not strictly less than the corresponding edge of $P$, it can be increased to strictly increase the area of $Q$, contradicting its optimality again.  If the remaining edge of $Q$ contains the circumcenter of $Q$, the reflection method as in Figure \ref{fig:iso} applies, as with Queen Dido's problem.  Thus $P$ and $Q$ are congruent by the uniqueness for the equality case. \qed
\end{proof}
\bigskip
\noindent {\bf Proof of Theorem \ref{Thm:main}:} We first start with the case of a triangle with side lengths $a, b, c$, where $c$ is the largest size. 

By a Theorem of Lexell, see \cite{enwiki:1248978010} and \cite{MR4405555}, we know that the only critical points for the area function of a spherical triangle, where the lengths $a,b$ are fixed, are when the area is at a maximum or minimum (counting the oriented area).  See Figure \ref{fig:Lexell} that shows this and a statement of Lexell's theorem.

 \begin{figure}[!htb]
    \centering
    \includegraphics[width=0.7\textwidth]{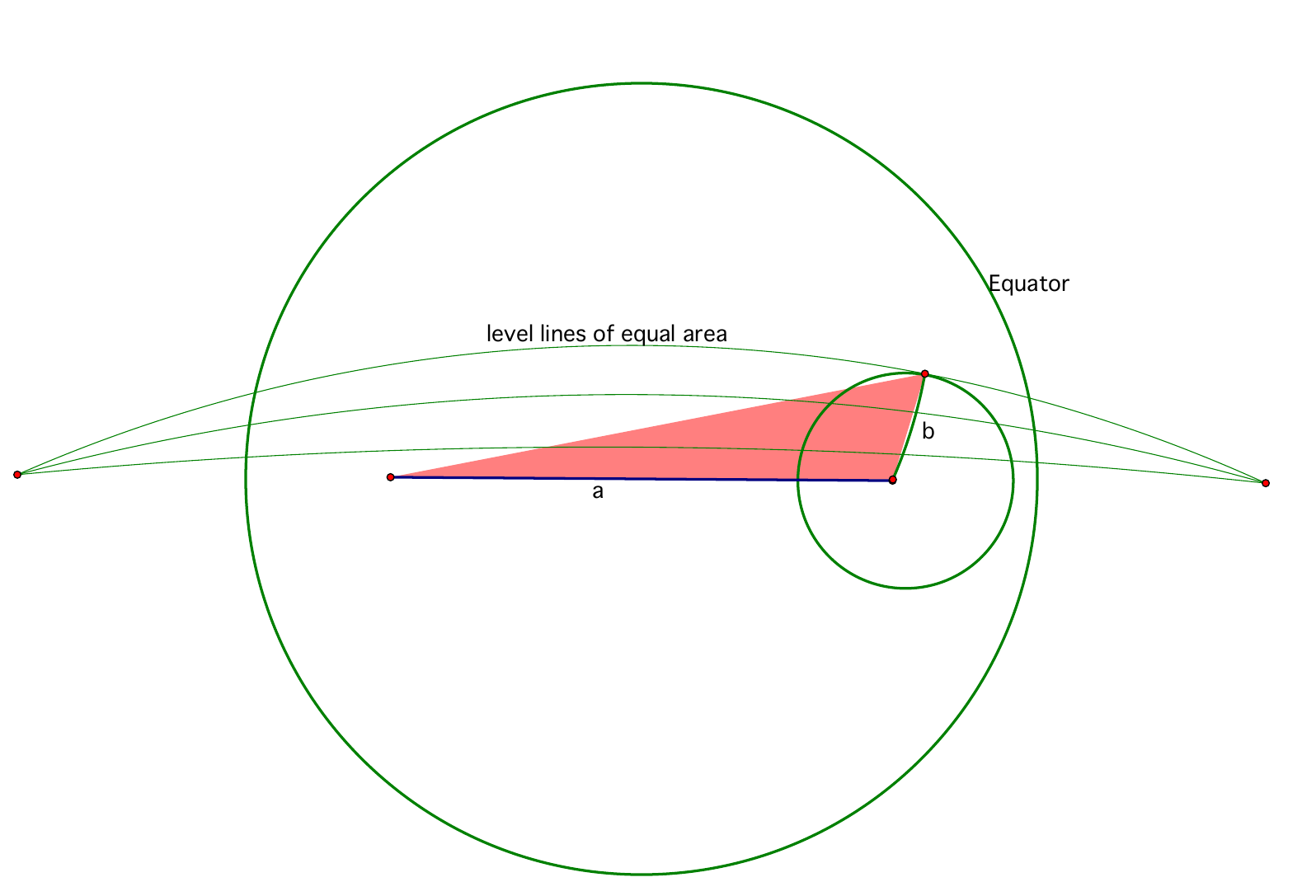}\captionsetup{labelsep=colon,margin=2cm}
         \caption{This shows the stereographic projection of  the level lines of the area of a geodesic triangle on the sphere with one fixed side of length $a$. They are circles, each of  which intersect the antipodes of the end points of the geodesic line segment of length $a$. 
         }\label{fig:Lexell}
    \end{figure}   

In the Queen Dido case, by reflecting about the line through $c$, we see that the resulting circumscribed quadrilateral with side lengths $a,b,a,b$ lies in the circumcircle about the center of $c$.  Thus by Theorem \ref{classic-iso}, $2a+2b < 2\pi$, thus $a+b < \pi$, and the Dido configuration is the only case when the area function is at a critical point. By rotating the $b$ edge about the edge $a$ we see that when the $c$ length is increased up to a maximum, where they lie on a geodesic with the circumcenter outside of the $a, b, c$ triangle. Thus as $c$ is increased from the Dido case, the area decreases and the circumcenter lies outside of the $a, b, c$ triangle on the $c$ side of the triangle.  

In the case that the $c$ is decreased, the circumcenter must lie on the inside of the $c$ side of the triangle.  It may lie outside of the triangle, but inside the $c$ side.

If the center of the circumscribed circle of the triangle lies strictly outside of the triangle, say on the $c$ side of the triangle, then $a+b < \pi$ and the previous analysis applies. Thus, decreasing the size of $c$ will increase the area of the triangle.

If the center of the circumscribed circle of the triangle lies inside the triangle, and thus not on the outside of the side $c$, then starting from the case when the $a$ and $b$ sides overlap we see that increasing the length of $c$ increases the area of the triangle by Lexell's Theorem until the circumcenter hits side $c$.

Next we consider the more general case of a spherical polygon circumscribed by a spherical circle.  If a side $c$ is to be increased, and the circumcenter is not outside $c$, then choose another point of the polygon, not one of the points of c, and consider the triangle formed by that point and the end points of $c$.  Fix those distances and increase the length of $c$, and extend it to both sides of the polygon.  This increases the total area of the polygon.  Then apply the isoperimetric Lemma \ref{fixed} to not decrease the area with those fixed lengths. See Figure  \ref{fig:Triangle-area} for non-Dido case.
 \begin{figure}[!htb]
    \centering
    \includegraphics[width=0.4\textwidth]{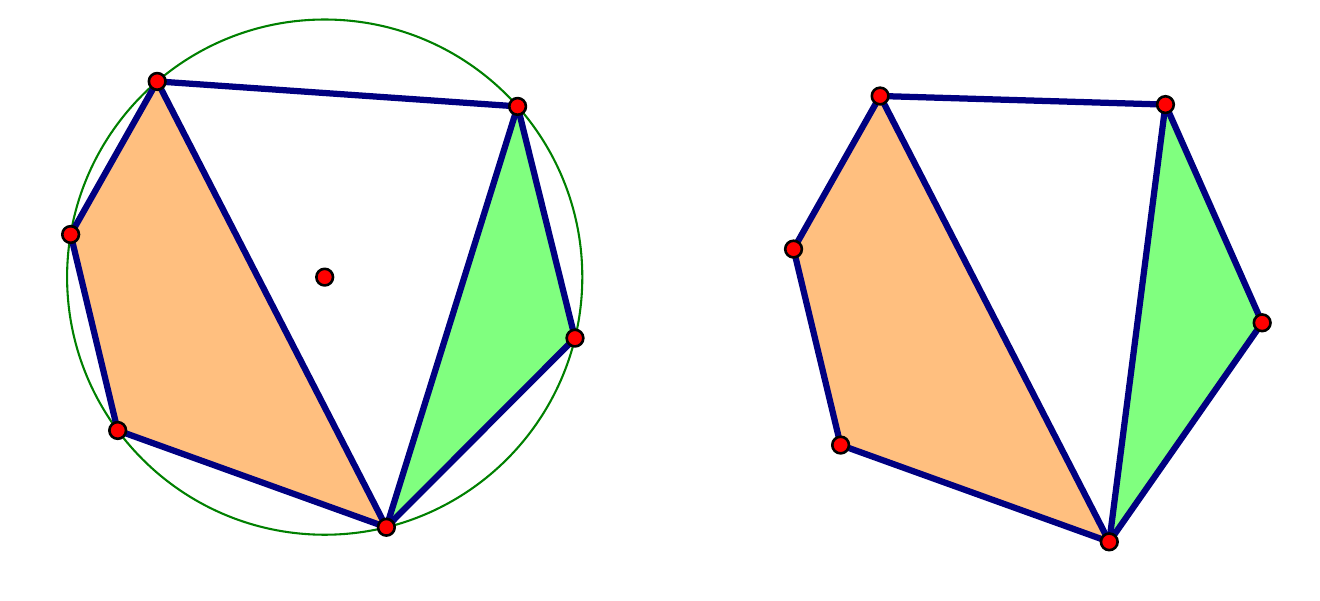}\captionsetup{labelsep=colon,margin=2cm}
         \caption{This shows that when a triangle has its circumcenter in its interior, then the length of the edge on the boundary of its containing polygon (on the top of the polygon here) determines  how the area of the whole polygon increases or decreases. The colored polygons have a fixed area in the two figures. When the circumcenter of a polygon lies inside the polygon, there always is a triangle, as in the figure, where decreasing the outer edge will decreas the area.
         }\label{fig:Triangle-area}
    \end{figure}

For the Dido case and when the circumcenter lies outside of the edges of the polygon, do the same as the internal case, but decreasing the length of $c$, increasing the area and applying Lemma \ref{fixed} again. \qed

 \section{The Connection to Rigidity: Spiderwebs on the Sphere}   

A result of Winter in \cite{Winter} shows that if one takes a convex polytope $P$, containing the origin $O$ in $\R^3$ and joins $O$ to each vertex of $P$ with a strut (that is not allowed to decrease its length) and forces the edges of $P$ to be cables (that are not allowed to increase in length), the resulting tensegrity is rigid in  $\R^3$.  It seems very likely that such a tensegrity is always pre-stress rigid, see \cite{Connelly-Guest} for a definition. It is also conjectured that some type of global rigidity should hold. In this paper, we will assume that the vertices from the center point are fixed at distance one. In other words, the vertices of $P$ lie on the unit sphere, and all the other possible configurations have their vertices on the unit sphere.

Here we propose a very special case of Winter's problem that has a very strong global property of rigidity that is analogous to the rigidity of convex polytopes as observed by Cauchy in 1813, as well as a spider web interpretation, where the conditions for the rigidity can be verified easily.  See \cite{Cauchy} and in \cite{Rigidity-Energy}.  In Winter's Theorem, it seems that the rigidity of the coned polytope is equivalent to the rigidity of the projection of the vertices and edges of the polytope into the unit sphere. The cable constraints are enforced by the corresponding geodesic distance constraints on the projected vertices. From the tensegrity point of view, here we are assuming that the edges from the center vertex is a bar rather than just a strut.   It turns out one can reverse the procedure to create a convex polytope from a rigid tensegrity on the surface of the sphere from implicit stresses that must exist from the rigidity on the sphere using Maxwell-Cremona Theory or other methods.  But that is not relevant to the results here. 

Suppose that $G$ is the graph of a convex $P$ polytope, whose vertices $p$ lie on the unit sphere $\mathbb{S}^2$ in $\R^3$. We assume that each pair of vertices of $p$ corresponding to an edge of $P$  is not antipodal. Extend this map of vertices of $P$  to its edges by a shortest geodesic arc in $\mathbb{S}^2$, which can be done uniquely since those vertices are not antipodal. This then defines a continuous map of the $1$-skeleton of $P$ to $\mathbb{S}^2$.  Next we assume that the sum  of  lengths of the geodesics corresponding to the boundary edges of each face of $P$ is strictly less than $2\pi$. By Lemma \ref{hemi}, the image of the boundary of each face of $P$ is contained in a hemisphere of $\mathbb{S}^2$.  Thus, the map can be extended to each face, uniquely topologically up to homotopy.  Call such a graph $G$ and configuration $p$ an \emph{edge-constrained} tensegrity $(G,p)$.
Thus, such an edge-constrained tensegrity $(G,p)$ has a well-defined topological degree as a topological map between topological spheres, $P$ and $\mathbb{S}^2$ (given their orientations).  Note that one can interpret the degree of the topological map as the total signed oriented area of the spherical polygons divided by $4\pi$.  

 With this in mind, we define a polytope $P=(p_1,\dots,p_n)$ is \emph{semi-globally rigid} with respect to a point $p_0$ in the interior of P, if every other realization, fixing $p_0$, $Q=(q_1,\dots,q_n)$ of $P$ with corresponding edge lengths, the cables, no longer (i.e. $|q_i-q_j| \le |p_i-p_j|$), and with corresponding strut lengths  $|q_i-p_0| \ge |p_i-p_0|$ from $p_0=0$ no shorter, and in the same homotopy class, then $Q$ is  congruent to $P$ by a rotation about $p_0$.

\begin{theorem}\label{Theorem:tensegrity}
   Let $P$ be a convex inscribed polytope in $\R^3$ (that is, all its vertices lie on a sphere) such that  the circumcenter of each face (the center of the circle through each vertex of the face) lies in its face or on the boundary of its face.  Then $P$ is semi-globally rigid. 
\end{theorem} 

\begin{proof} Note that the topological degree of the possible realizations of the points $p$ is well-defined by the  length constraints of the edges, and the faces of $P$ are circumscribed by a circle since they lie in the intersection of the faces and the circumscribing sphere. By Corollary \ref{lemma:congruence}, in the configuration $(G,p)$, each face is at a maximum area, given the edge length constraints, and that maximum configuration is unique up to congruence. But the spherical sum of the areas of the faces for a unit sphere is $4\pi$. Thus each face of any other configuration $q$, satisfying the edge length constraints, is congruent to the corresponding face of $P$.  So $q$ is congruent to $p$. \qed
\end{proof}

\bigskip

We recast this property in the context of geodesics on the unit sphere.
\begin{definition}
Suppose that a $3$-connected planar graph $G$ is mapped continuously to the unit sphere $\mathbb{S}^2$ with its edges as geodesics.  We say this map is a \emph{tight} realization if each geodesic edge is less than $\pi$ and the length of the perimeter of each face of $G$ is less than $2\pi$. We say the map is \emph{weakly tight} when the ``less than" inequalities are replaced with ``less than or equal to".

A \emph{face} of $G$ is a non-separating simple closed curve in $G$.
\end{definition}

\begin{lemma}
    If a tight realization $p$ of $G$, as above,  is embedded, then any homotopy of tight realizations of $G$ corresponds to a homotopy class of $\mathbb{S}^2 \rightarrow \mathbb{S}^2$.  Then the degree of a tight realization is defined as the degree of the map of $\mathbb{S}^2$.
\end{lemma}
\begin{proof}By Lemma \ref{hemi} the interior of each face lies in an open hemisphere, and the same is true for the image. Thus, the map extends to the interior of each face since the open hemisphere is simply connected. \qed
\end{proof}

\bigskip

The connection to polytopes is the following.
\begin{lemma}\label{lemma:tight} The projection of the boundary of a convex polytope $P$ from a point $p_0$ in the interior of $P$ onto the unit sphere $\mathbb{S}^2$ about $p_0$ is tight.     
\end{lemma}
\begin{proof} The plane of each face $F$ of $P$ does not contain $p_0$.  The parallel plane through $p_0$ defines a geodesic circle in $\mathbb{S}^2$, which is the boundary of a hemisphere $H$ containing the projected vertices and edges of $F$ in its interior, where the  edges of $F$ project to shortest geodesics on $\mathbb{S}^2$. Thus each  geodesic has length less than $\pi$. 

\begin{figure}[!htb]
    \centering
    \includegraphics[width=0.3\textwidth]{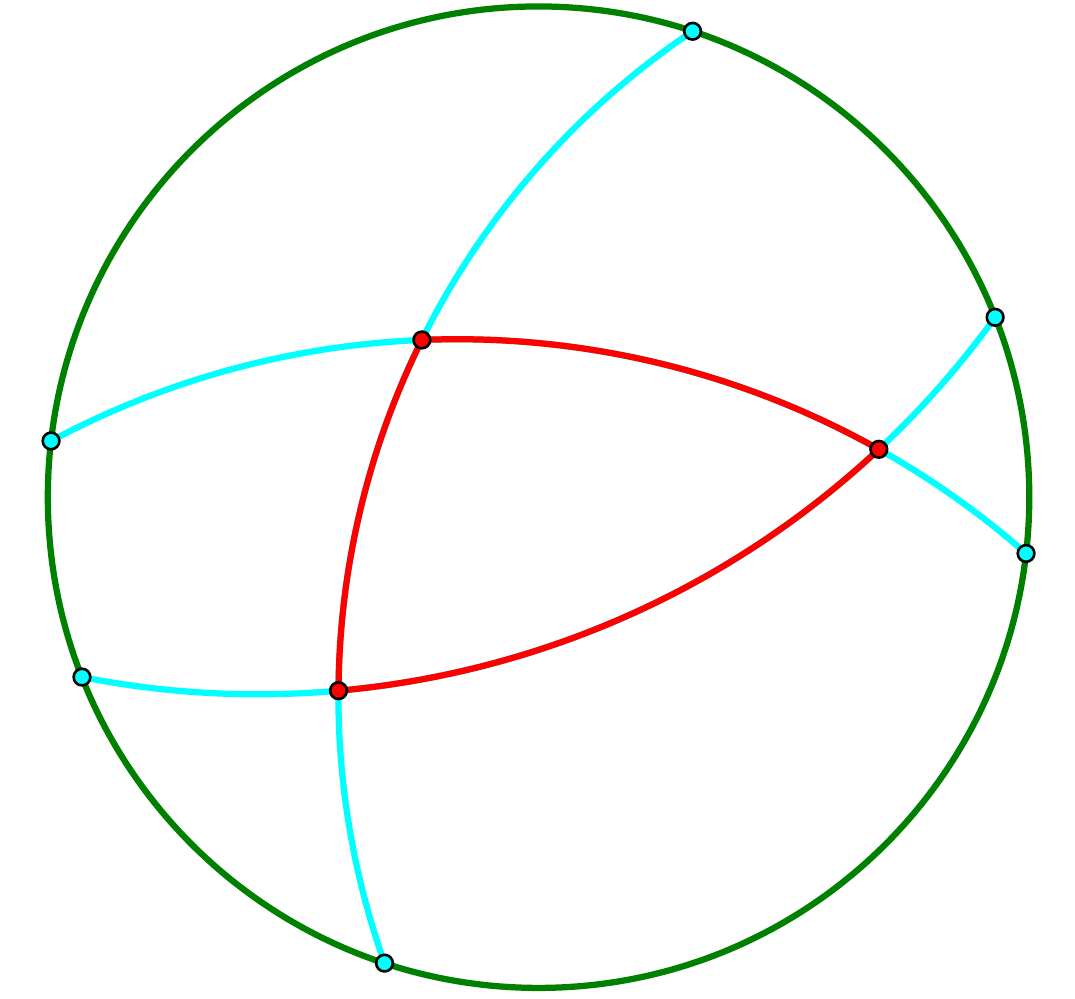}\captionsetup{labelsep=colon,margin=2cm}
\caption{This shows the stereographic projection of a geodesic polygon triangle in the image of a hemisphere.}\label{fig:Nested-convex}
    \end{figure}

To show that the length of the perimeter of each geodesic polygon is less than $2\pi$, we note that, at least in the plane, when one convex set strictly contains another, the perimeter of the larger polygon is greater than the perimeter of the smaller one.  The projection from $p_0$ of a face $F$ into the open hemisphere $H$ is spherically convex, and the closed hemisphere $\bar{H}$ is convex in the sense that between any pair of points in 
$\bar{H}$, there is a shortest geodesic between them in $\bar{H}$, but it may not be unique. Nevertheless, the length of the perimeter of $\bar{H}$ is $2\pi$ and the length of the perimeter of any polygon in $H$ is strictly less than $2\pi$.  Figure \ref{fig:Nested-convex} shows how this is seen.  For each edge of a geodesic polygon, extend it to a geodesic circle that intersects the boundary of geodesic of $H$ in two antipodal points.  The intersection of $H$ and the hemisphere containing the blue geodesic, and the polygon, is another spherical polygon containing the spherical polygon that is the image of $F$. The perimeter of this region of the sphere (a ``lune") still has perimeter $2\pi$.  But when this ``shaving" of $H$ is done again for another edge, the length of the perimeters of the resulting exterior polygon then decreases strictly until it gets to the final polygon, which is the image of $F$ in $H$. This shows that the projection of the edges of the graph of $P$ is tight. \qed
\end{proof}

\bigskip

Notice, also, that such a circumscribed polytope corresponds to a circle covering of the sphere, where the vertices correspond to the intersection of at least three circles as in Figure \ref{fig:skew-cube}
    \begin{figure}
\centering\includegraphics[width=0.5\textwidth]{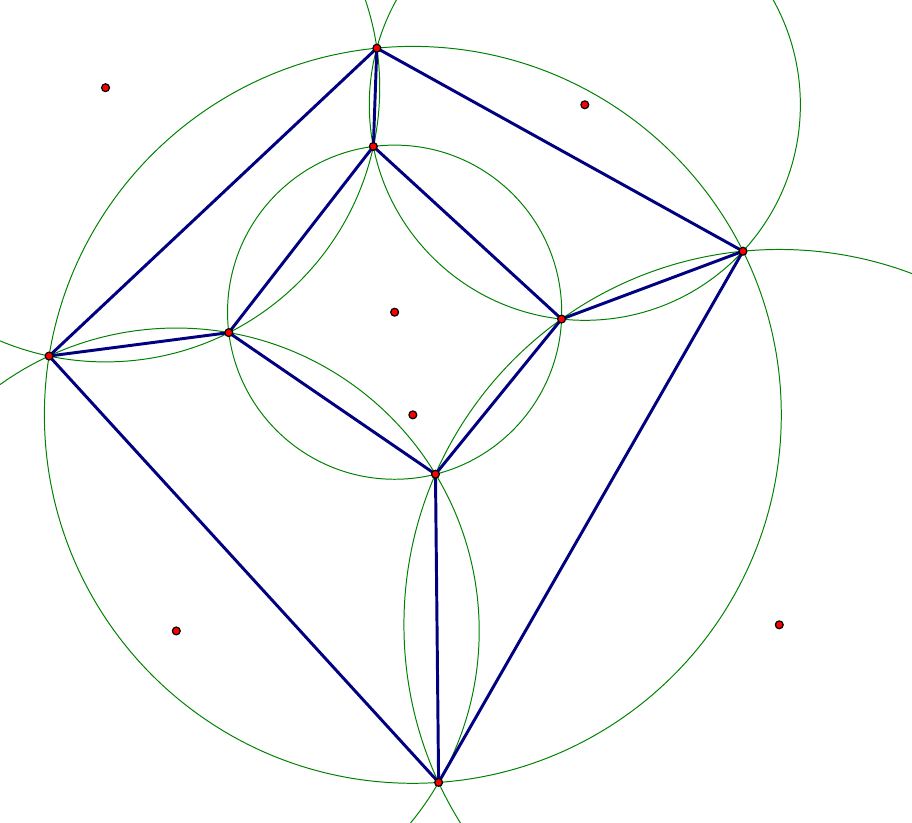}\captionsetup{labelsep=colon,margin=2cm}
         \caption{This shows a construction of the analogue of a polytope with the same graph as the cube, but in the plane.  If this is done in the sphere, we get a polytope with the same edge graph as the cube, with six flat quadrilateral planar faces.  Here the centers of the circles of the side quadrilaterals are outside the  quadrilateral faces.  The centers can be be made inside on the sphere.}
         \label{fig:skew-cube}
    \end{figure}   
 See \cite{Padrol2016} for information about which graphs correspond to the graphs of a circumscribed polytope.

In \cite{Winter}, it is shown that such a convex polytope corresponding to a tensegrity graph is rigid in $\R^3$ regardless of whether it is circumscribed.  But there is strong evidence that it is always pre-stress stable (see \cite{stress-flex} and \cite{Connelly-Guest} for the definition of pre-stress stable), which is quite natural. For the circumscribed case, a proof that it is pre-stress stable is coming.

    When a bar framework or tensegrity $(G,\p)$ is infinitesimally rigid, it implies that when there is another configuration $\q$, that is sufficiently close to $\p$ and  that satisfies the distance constraints, it is congruent to $\p$.  Often it is left at that, with no estimate, a-priori, of how close.  One important exception is Cauchy's original theorem about triangulated polyhedra (See \cite{Connelly-Guest} or \cite{Gluck} for more modern proofs.). In that case, one assumes that both configurations are convex, and that allows a reasonable estimate of how close $\q$ can be to $\p$.  In the case when the tensegrity $(G,|p)$ corresponds to a circumscribed polytope with circumcenters in their faces, as with Theorem \ref{Theorem:tensegrity}, one can make fairly strong estimates of how close the configuration $\q$ needs to be to $\p$.  For example, for a regular cube with vertices circumscribed in the unit sphere, if each of $q_i$ is within $\pi/4$ of $p_i$, each face of $Q$, when projected into the sphere, is inside of the hemisphere of the corresponding face of $P$.  So the degree of the corresponding map of the sphere is forced to be one.  Thus $Q$ is congruent to $P$ by a rotation. 

We mention one question as to whether this is \emph{always} a unique realization of a stressed graph on the unit sphere.

\begin{conjecture}\label{Conj:unique}  Given a vertex $3$-connected planar graph embedded in the unit radius sphere, such that each edge is a geodesic of length less than $\pi$ and each face has its perimeter of total length less than $2\pi$, and there is an equilibrium stress positive on each edge, then that  configuration is unique up to rotation, where each geodesic edge is constrained to not get larger in the class of degree one maps of the sphere to itself, where the map to the sphere is determined as an extension of the origonal sphere by extending the map on the origonal realization.
    \end{conjecture}

    So far we have not found any counter-example.  All the examples of such spider webs, we can show, by one method or another, that the configuration is unique up to rotation in its homotopy class, or we don't know whether it is unique in its homotopy class.

In the spirit of finding semi-global geodesic configurations, for each $3$-connected embedded graph $(G,\p)$ on the unit sphere, whose topological degree is one, and where the length of each facial perimeter is less than $2\pi$, we define a partial order $\le$.  We say that $(G,\p) \le (G,\q)$ if, for each edge $e$ of $G$, the geodesic length of $e$ in $\p$ is less than or equal to the geodesic length of the edge $e$ in $\q$.

\begin{conjecture}\label{Conj:degree}
Any locally minimal $(G,\p)$ is unique in its homotopy class up to rotation with respect to the partial order.
\end{conjecture}

\begin{remark}
In the plane, one of the methods of indicating when a $3$-connected planar graph $G$ has an equilibrium stress is to exhibit what is called a ``reciprocal diagram" $H$, whose faces, edges, and vertices correspond to the vertices, edges, and faces of $G$, and corresponding edges are perpendicular.  For example, configurations of the  graphs of a cube and octahedron can be reciprocals.  It is interesting on the unit sphere for an inscribed convex polytope, when the circumcenter of each face, projected out to the sphere,  is taken as the corresponding vertex of the reciprocal, the corresponding (geodesic) edges are perpendicular. 

Indeed in \cite{Verdier} and \cite{Lovasz} it is shown that the spherical spiderweb that is $3$-connected and planar with a positive stress corresponds to the projection of a convex polytope  along with its  corresponding reciprocal spiderweb, where the two polytopes are polar duals in $3$-space.  

\end{remark}

\section{Examples}

Any inscribed polytope $P$ in $\R^3$ with circumcenters inside its corresponding faces will satisfy the conditions of Theorem \ref{Theorem:tensegrity} and be congruent to $P$ by rotation. For example, the regular cube with its cone point in the center of symmetry will be unique in the degree one realizations in the unit sphere subject to its cable constraints.

\bigskip

 Figure \ref{fig:Slack-cube} shows a side view of the edge graph of a cube, where the vertices lie on a sphere.
The red edges are slack in the sense that the stress on those interval edges must be zero. Without the red edges, the three geodesics from $N$ to $\mathbb{S}$ can be rotated so that all of the geodesics are in one closed hemisphere and then homotoped to a point. Note that without the red cables, the notion of a homotopy class is not well-defined, since the lengths of regions on the sphere are not strictly less than $2\pi$, not to mention that the graph is not $3$-connected. 

When the red cables are inserted all the motions are stopped and it is rigid as a cabled framework in $R^3$.  But further, when the red cables are inserted, we can apply Lemma \ref{lemma:closed-hemi} to exclude the case when all the regions are in one (closed) hemisphere, since the notion of degree of a map of the sphere is well-defined in that case.  This situation is very similar to the notion of $n$-step universal rigidity in \cite{universal}.  The first step in the case of Figure \ref{fig:Slack-cube} is when the half-geodesics from $N$ to $S$ are stressed, but that part alone is not rigid as mentioned above.  When the slack cables are stressed, it is the second step, and it has its own sense of being in equilbrium. 

 \begin{figure}
\centering\includegraphics[width=0.23\textwidth]{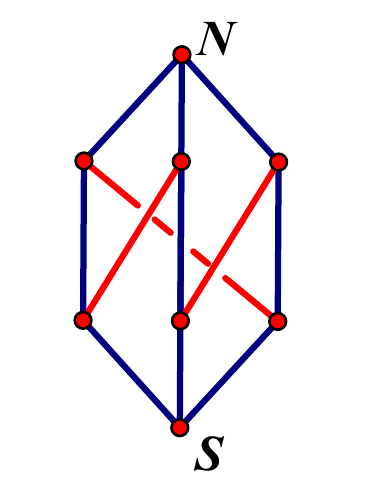}\captionsetup{labelsep=colon,margin=2cm}
         \caption{This shows the projection from the side of a framework, where the edges are shown as straight line segments in $\R^3$ instead of geodesics on the unit sphere.  The four vertices on each the black edge paths from $N$ to $\mathbb{S}$ are all in a single plane including the origin, as well as $N$ and $\mathbb{S}$.}
         \label{fig:Slack-cube}
    \end{figure} 

    \bigskip

Even when just the graph of the polytope is given, it may not be possible to find any realization as a convex polytope where all the vertices lie on a sphere, i.e. where the polytope is inscribed in a sphere.  Figure \ref{fig:Triakis} is one such example.  Note that in this figure, when all the vertices are vertices of a cube, Theorem \ref{Theorem:tensegrity} applies since the circumcenter of each of the triangular faces lies on the boundary of their triangular right-angled face. For a careful discussion of when a polytope is inscribable or, dually circumscribable, see \cite{inscribable}.
\begin{figure}
\centering\includegraphics[width=0.4\textwidth]{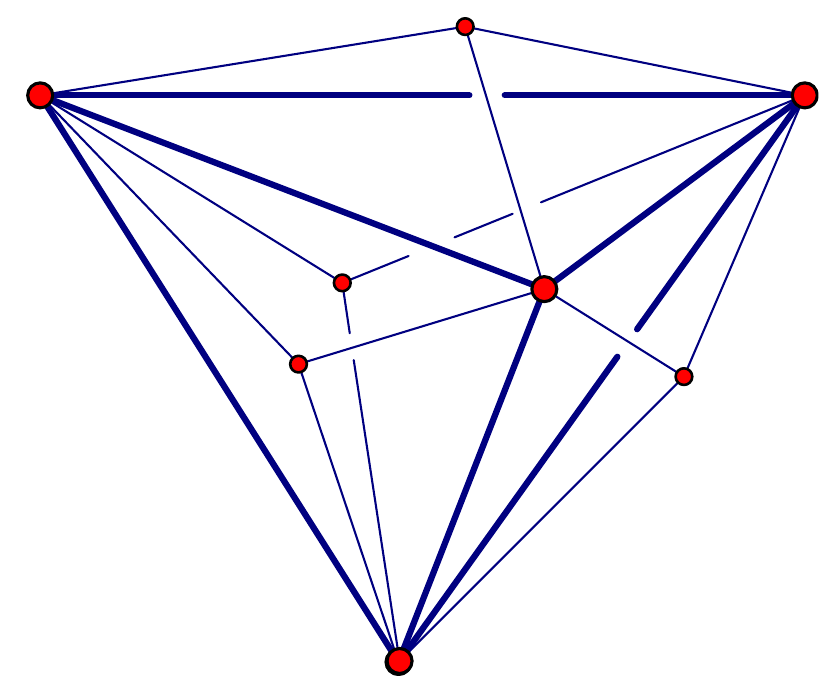}\captionsetup{labelsep=colon,margin=2cm}
         \caption{This is the triakis tetrahedron, where a tetrahedron (with large vertices) has another tetrahedron (with lighter weighed edges) attached to each of its faces, creating another convex polytope. See the Wikipedia article triakis tetrahedron for another figure.}
         \label{fig:Triakis}
    \end{figure}

Note that even when a  polytope is inscribed in a sphere, with all faces triangles, it may not be possible to apply Theorem \ref{Theorem:tensegrity}, which assumes that the center of the inscribed polygonal faces be inside or on the boundary of that face.  If one projects the center of the sphere orthogonally into the plane of the triangle, that condition implies that each triangle of the polytope is acute.  In Winter's result \cite{Winter} that is not necessary, and the uniqueness also follows since a triangle is always planar.

    \bigskip

Another example is when there is a corresponding map to the sphere of higher degree than one.  Namely, Figure \ref{fig:Pentagrams} shows two parallel pentagrams in $\mathbf{R}^3$ joined by $5$ rectangular faces.  When the pentagrams are extended to maps to the sphere, and similarly for the rectangular faces, they correspond to a degree two map of $\mathbf{S}^2$ to itself. Because of the symmetry, there is an equilibrium stress that is  positive on all the cable edges.  There is an ``unwrapping map'' that shortens all the edge lengths, but that is a degree one map from the sphere to itself.
\begin{figure}[!htb]
    \centering
    \includegraphics[width=0.4\textwidth]{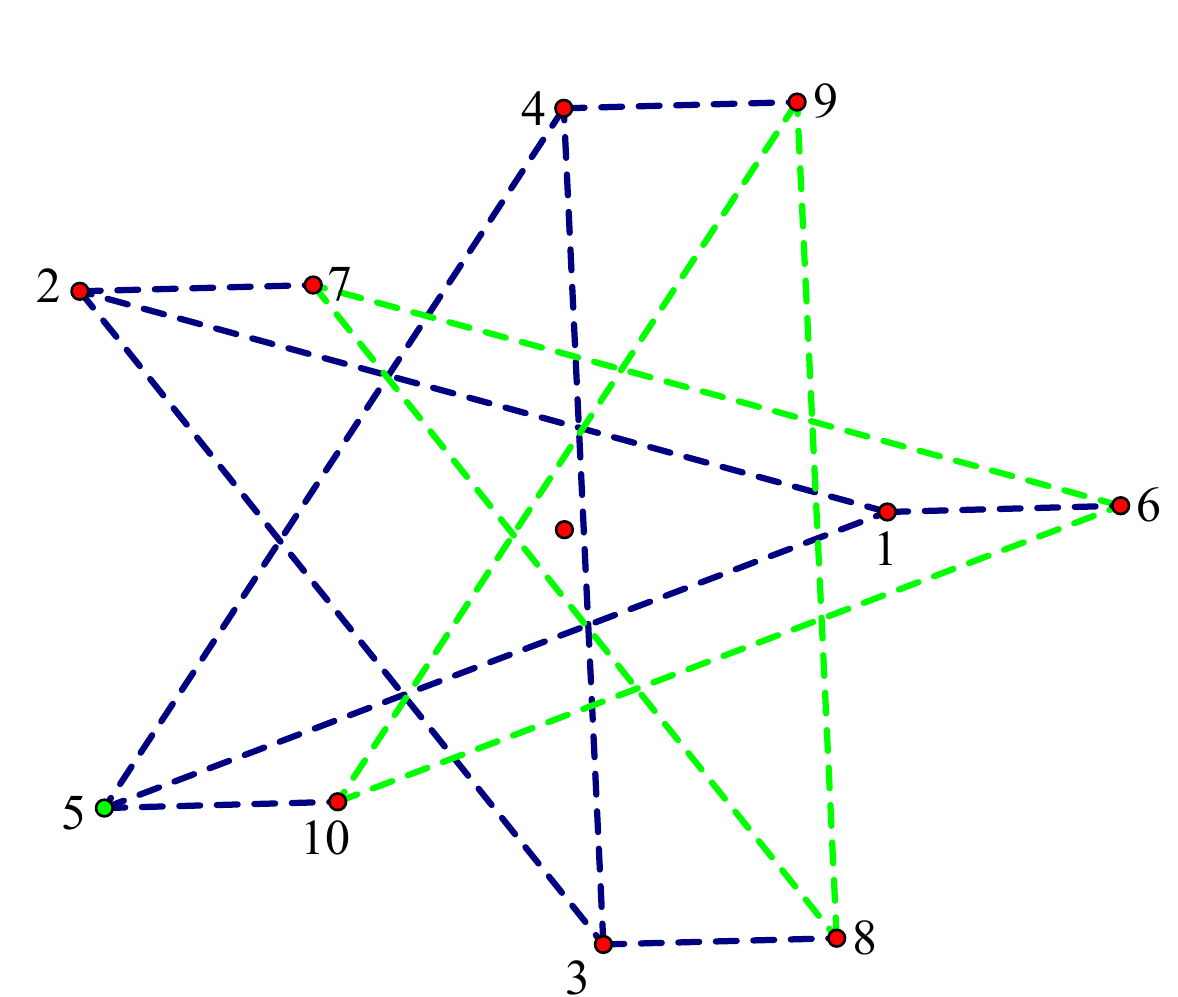}\captionsetup{labelsep=colon,margin=2cm}
         \caption{This shows the pentagrams and their connecting edges that correspond to a degree $2$ map of $\mathbf{S}^2$.
         }\label{fig:Pentagrams}
    \end{figure}

\bigskip

A lot of the ideas for the spider webs on the sphere are similar to the situation for spider webs on compact surfaces of constant negative curvature.  See \cite{Harmonic-Maps} or 

\cite{Iso-surfaces}, for example.  

\bigskip

We thank Tom Banchoff for discussions over the years past about the isoperimetric problem.

\newpage

\bibliographystyle{abbrvnat}
\bibliography{Iso.bib}

\end{document}